\documentclass[12pt]{amsart}
\usepackage{amssymb,amsmath,amsthm,mathrsfs,multirow,enumerate}

\usepackage{graphicx} 
\usepackage[all]{xy}
\usepackage{braket}
\usepackage[top=1.5 in, bottom=1.5 in, left=1.3 in, right=1.3 in]{geometry}
\allowdisplaybreaks
\numberwithin{equation}{section} 

\theoremstyle{plain}
\newtheorem{thm}{Theorem}[section]
\newtheorem{cor}[thm]{Corollary}
\newtheorem{lemma}[thm]{Lemma}
\newtheorem{prop}[thm]{Proposition}

\theoremstyle{definition}

\newtheorem*{rmk}{Remark}

\newcommand{\diag}{\mathop{\mathrm{diag}}}

\newcommand{\PSL}{\mathrm{PSL}}

\newcommand{\tr}{\mathrm{Tr}}
\newcommand{\ad}{\mathrm{ad}}

\newcommand{\bbm}{\begin{bmatrix}}
\newcommand{\ebm}{\end{bmatrix}}
\DeclareMathOperator{\Aut}{\rm Aut}

\DeclareMathOperator{\cha}{\rm char}

\begin{document}

\title{Classical orthogonal decomposition of a modular $\mathfrak{sl}_n$}

\author{Yotsanan Meemark and Songpon Sriwongsa}
%\thanks{*Corresponding Author}

\address{Yotsanan Meemark \\ Department of Mathematics and Computer Scienc \\ Faculty of Science \\ Chulalongkorn University \\ Bangkok 10330, Thailand}
\email{\tt yotsanan.m@chula.ac.th}

\address{Songpon Sriwongsa \\ Department of Mathematics \\ Faculty of Science \\ King Mongkut's University of Technology Thonburi (KMUTT) \\ 126 Pracha Uthit Rd.\\ Bang Mod, Thung Khru \\ Bangkok 10140, Thailand}
\email {\tt songpon.sri@kmutt.ac.th}

%\keywords{Classical Cartan subalgebras, Modular Lie algebras, Orthogonal decomposition}

%\subjclass[2020]{Primary: 17B50; Secondary: 13M05}
\thanks{Note added in this arXiv version: we correct the conclusion about the uniqueness of COD of $\mathfrak{sl}_3$ over a finite field in Section 3 of the published version [J. Pure Appl. Algebra 228 (2024) 107721] by adding Theorem \ref{unique3} to this version. Additionally, we add the input of 
$m[i], n[i], k[i], l[i], s[i], t[i], x[i]$ and $y[i]$ in Appendix A: $A.1.$ Mathematica code for checking $J_3(1, z)$ and $J_3(1, 1)$.}

\begin{abstract}
An orthogonal decomposition problem of Lie algebras over the complex numbers has been studied since the 1980s. It has many applications and relations to other areas of mathematics and sciences. In this paper, we consider this decomposition problem over a field of prime characteristic. We define a classical orthogonal decomposition of a modular Lie algebra and construct it for $\mathfrak{sl}_n$ under certain sufficient conditions. Additionally, we provide more detailed analysis of the problem when $n = 2$ and $3$. 
\end{abstract}

\maketitle

\section{Introduction}\label{intro}
Throughout the paper, we denote the special linear Lie algebra of order $n$ over a field $\mathbb{F}$ by $\mathfrak{sl}_n(\mathbb{F})$. 

 Over the field of complex numbers $\mathbb{C}$, an {\it orthogonal decomposition} (OD) of a Lie algebra is a decomposition (as a vector space) into a direct sum of its Cartan subalgebras which are pairwise orthogonal with respect to the Killing form. One interesting example of OD was studied by J. G. Thompson in \cite{T76, Th76} where an OD of the Lie algebra of type $E_8$ was discovered while constructing the sporadic simple group $F_3$. The theory of OD of simple Lie algebras over $\mathbb{C}$ was later developed by Kostrikin et al. \cite{KK81, KK83, KK84, KT94}. The problem of OD has attracted further attention due to its applications and its relations to other fields. One important example is the study of mutually unbiased bases (MUBs) in $\mathbb{C}^n$ which has an application in quantum information theory \cite{BS07, DE10, R09}. In particular, a connection between the existence problem of OD of $\mathfrak{sl}_n(\mathbb{C})$ and the construction of maximal collections of MUBs was established in \cite{BS07}. 
 There is a famous conjecture:  $\mathfrak{sl}_n(\mathbb{C})$ has an OD if and only if $n$ is a prime power \cite{KT94}. It is known as the Winnie-the-Pooh conjecture, due to a play on words found in Zahoder's translation of Milne's famous children's book ``Winnie-the-Pooh" into Russian. 
 The ``only if" part of this conjecture remains unsolved and wide open, even for the case of the first non-prime power $n = 6$. Some developments of the OD problem for $\mathfrak{sl}_6(\mathbb{C})$ can be seen in \cite{BZ16, KZ21} and references therein. Indeed, the conjecture has an impact on the problem of MUBs. 
 Recently, the Winnie-the-Pooh problem has been related to an algebraic combinatorics problem (see \cite{TW20}). 
 In addition to the aforementioned references, there are other works related to our topic that are worth mentioning.
First, a review about MUBs and applications of OD in homological algebra and topology can be found in \cite{BYZ21}. The results regarding the description of MUBs in dimension $7$ were presented in \cite{N06, ZK19}. The reader can also find the relationship between the OD problem and symplectic geometry in \cite{BZ19}. Furthermore, the application of the $p$-adic model of quantum mechanics in studying of MUBs and OD has been published in \cite{Z23}.
 
 One can raise a question about the orthogonal decomposition problem for the Lie algebra $\mathfrak{sl}_n$ over a field of positive characteristic or any other commutative ring. For recent works on the case of finite commutative rings, see \cite{S23, SZ20}. This particular problem was also mentioned on page 54 of \cite{KT94}. It was claimed that if $\mathfrak{sl}_n(\mathbb{C})$ has an OD, then the same property is evidently satisfied by $\mathfrak{sl}_n(\mathbb{F})$, where $\mathbb{F}$ is an algebraically closed field with $\cha (\mathbb{F}) \nmid n$, and the question arose: Is the converse true? Of course, $\mathfrak{sl}_n(\mathbb{F})$ and $\mathfrak{sl}_n(\mathbb{C})$ share many of the same properties. However, without assuming the algebraically closed property, the problem turns out to be more complicated. 
 In this paper, we extend the definition of an OD for a modular Lie algebra over an arbitrary field $\mathbb{F}$ of positive characteristic (not necessary algebraically closed) and construct it for $\mathfrak{sl}_n(\mathbb{F})$ under some sufficient conditions. Then we characterize these fields allowing the decomposition to occur when $n = 2$ and $3$. Moreover, we analyze the uniqueness problem for these cases. 
 
 Some properties of a modular Lie algebra do not hold like in a Lie algebra over $\mathbb{C}$, in particular, the Cartan subalgebras' properties. For instance, all Cartan subalgebras of a semisimple Lie algebra over $\mathbb{C}$ are conjugate and abelian \cite{H72}, but this is not true for our setting here. It is worth mentioning that there is a finite dimensional semisimple Lie algebra over a finite field,
which has a non-abelian Cartan subalgebra \cite{D72}.
 Thus, we shall modify the definition of an OD for the modular case as follows.
  
 Let $\mathbb{F}$ be a field of positive characteristic $p > 3$.
 Let $\mathfrak{L}$ be a Lie algebra over $\mathbb{F}$. Recall that a subalgebra $H$ of $\mathfrak{L}$ is a {\it Cartan subalgebra} if it is a nilpotent subalgebra which is a self-normalizer. In this work, we also require all Cartan subalgebra to be ``classical". The details are described in the following paragraph.   
 
 We recall from \cite{S67} that  a Lie algebra $\mathfrak{L}$ over $\mathbb{F}$ is {\it classical} if:
\begin{enumerate}[(i)]
	\item the center of $\mathfrak{L}$ is zero;
	\item $[\mathfrak{L}, \mathfrak{L}] = \mathfrak{L}$;
	\item $\mathfrak{L}$ has an abelian Cartan subalgebra $H$ which satisfies:
	\begin{enumerate}[(a)]
		\item $\mathfrak{L} = \oplus \mathfrak{L}_\alpha$, where $[x, h] = \alpha(h)x$ for all $x \in \mathfrak{L}_\alpha$ and $h \in H$;
		\item if $\alpha \neq 0$ is a root, then $[\mathfrak{L}_\alpha, \mathfrak{L}_{-\alpha}]$ is one-dimensional;
		\item if $\alpha$ and $\beta$ are roots with $\beta \neq 0$, then not all $\alpha + k\beta$ are roots, where $1 \leq k \leq p - 1$.
	\end{enumerate}
Such an $H$ possessing properties (a), (b) and (c) is called a {\it classical} Cartan subalgebra.
\end{enumerate}
Note that all classical Cartan subalgebras of a classical Lie algebra are conjugate, in particular, any two classical Cartan subalgebras of a classical $\mathfrak{sl}_n(\mathbb{F})$ are isomorphic by an element in the projective special linear group $\PSL_n(\mathbb{F})$ \cite{S67}.  We remark that not all Cartan subalgebras need to be classical, even when they are abelian. For example, the Cartan subalgebra $\left<
\begin{pmatrix}
    0 & 1 \\
    -1 & 0
\end{pmatrix}
\right >_{\mathbb{Z}_7}$ is abelian but not classical in $\mathfrak{sl_2(\mathbb{Z}_7)}$ because $\sqrt{-1}$ is invalid in $\mathbb{Z}_7$. Therefore, the adjoint action of the matrix is not semisimple, that is, $\mathfrak{sl_2(\mathbb{Z}_7)}$ does not have a  root space decomposition relative to this subalgebra.

Here, we define the orthogonal decomposition of 
\[
\mathfrak{sl}_n(\mathbb{F}) = \{n \times n \text{ traceless matrices over } \mathbb{F} \}
\]
to be {\it classical} if all of its components are classical, for this case, we denote the decomposition by COD. Recall that the orthogonality is defined via the Killing form: $K(A, B) := Tr (\ad A \cdot \ad B)$ and for the Lie algebra $\mathfrak{sl}_n$
\[
K(A, B) = 2n Tr(AB),
\]
where $Tr$ is the trace of a matrix.
Therefore, a COD of $\mathfrak{sl}_n(\mathbb{F})$ is a decomposition
\[
\mathfrak{sl}_n(\mathbb{F}) = H_0 \oplus H_1 \oplus \cdots \oplus H_n,
\]
where $H_i$'s are pairwise orthogonal classical Cartan subalgebras of $\mathfrak{sl}_n(\mathbb{F})$ with respect to $K$. 
The number $n$ is the Coxeter number of type $A_{n - 1}$. This holds because all components are conjugate. 

\begin{rmk}
    It is possible to find a field $\mathbb{F}$ such that $\mathfrak{sl}_n(\mathbb{F})$ can be orthogonally decomposed into a direct sum of Cartan subalgebras where at least one Cartan subalgebra is non-classical. We give one example here: 
    \[
    \mathfrak{sl_2(\mathbb{Z}_7)} = \left<
\begin{pmatrix}
    1 & 0\\
    0 & -1
\end{pmatrix}
\right >_{\mathbb{Z}_7} \oplus \left<
\begin{pmatrix}
    0 & 1 \\
    -1 & 0
\end{pmatrix}
\right >_{\mathbb{Z}_7} \oplus \left<
\begin{pmatrix}
    0 & 1 \\
    1 & 0
\end{pmatrix}
\right >_{\mathbb{Z}_7}.
    \]
    As discussed above, the second component is not classical. 
\end{rmk}

{\bf The organization of the paper.} In Section \ref{COD}, we derive the main result by constructing a COD of $\mathfrak{sl}_n(\mathbb{F})$ under some sufficient assumptions. We provide two corollaries for the cases when $\mathbb{F}$ is algebraically closed or finite. In Section \ref{characterization}, we characterize positive characteristic fields $\mathbb{F}$ such that $\mathfrak{sl}_n(\mathbb{F}), n = 2, 3$, has a COD. Moreover, we show that a COD of $\mathfrak{sl}_n(\mathbb{F})$ for $n = 2, 3$ where $\mathbb{F}$ is algebraically closed or finite, is unique up to conjugacy. 
 
\section{A construction of COD}\label{COD}

One construction of an OD of $\mathfrak{sl}_n(\mathbb{C})$, when $n = p^m$ is a prime power, is called a $J$-decomposition. It relies on a diagonal matrix consisting of $1$ and primitive $p$th roots of unity and a permutation matrix \cite{KT94}. The proof there was done by using Lie's theorem to verify that the constructed decomposition is an OD. However, the Lie's theorem is invalid for the modular case in general. Furthermore, all the constructed components are immediately classical in the complex case. 

In the following result, we also use the same matrices as mentioned for the construction of a COD, but the verification must be modified for this modular case. In fact, the classical property of all components needs to be shown. 

\begin{thm}\label{main1}
	Let $\mathbb{F}$ be a field of positive characteristic and let $n = p^m$ be a prime power. Assume that $\cha(\mathbb{F}) \neq 2, 3$ and $p$. If 
	\begin{enumerate}
		\item $p = 2$ and $-1$ is a square in $\mathbb{F}$ or
		\item $p > 2$ and $\mathbb{F}$ contains a primitive $p$th root of unity,
	\end{enumerate}
	 then $\mathfrak{sl}_n(\mathbb{F})$ has a COD. 
\end{thm}
\begin{proof}
We begin with the case $m = 1$. For $p = 2$, if $-1$ is a square in $\mathbb{F}$, then it is clear that $ \mathfrak{sl}_2(\mathbb{F})$ has a COD
\[
\mathfrak{sl}_2(\mathbb{F}) =\bigg \langle 
\begin{pmatrix}
1 & 0 \\
0 & -1
\end{pmatrix}
\bigg \rangle_{\mathbb{F}}
\oplus
\bigg \langle 
\begin{pmatrix}
0 & 1 \\
-1 & 0
\end{pmatrix}
\bigg \rangle_{\mathbb{F}}
\oplus
\bigg \langle 
\begin{pmatrix}
0 & 1 \\
1 & 0
\end{pmatrix}
\bigg \rangle_{\mathbb{F}}.
\]
Indeed, the diagonal one is a classical Cartan subalgebra and clearly conjugate to the last component. The second one is also a classical Cartan subalgebra due to the availability of $\sqrt{-1}$. Now, assume that $p > 2$ and  $\mathbb{F}$ contains a primitive $p$th root of unity $u$.
Let \[
D = \diag(1, u, \ldots, u^{p-1}) \text{ \ and \ } 
P =
\begin{pmatrix}
0 & 0 & \cdots & 0 & 1 \\
1 & 0 & \cdots & 0 & 0 \\
0 & 1 & \cdots & 0 & 0 \\
\vdots & \vdots & \ddots & \vdots & \vdots \\
0 & 0 & \cdots & 1 & 0
\end{pmatrix}.
\] 
Thus, $D$ and $P$ are matrices in $\mathfrak{sl}_p(\mathbb{F})$ and $p$ is the smallest positive integer such that $D^p = P^p = I_p$. For any $a, b \in \mathbb{Z}_p$, let $J_{(a, b)} = D^a P^b$. We have
\begin{align}\label{trzero}
 \tr (J_{(a, b)}) = 0 \Longleftrightarrow (a, b) \neq (0, 0)
 \end{align}
 and
\begin{align}\label{commuPD}
P^bD^a = u^{-ab}D^aP^b.
\end{align}
The last equation implies 
\begin{align}\label{brac}
J_{(a, b)}J_{(c, d)} &= u^{-bc}J_{(a + c, b + d)} \text{ \ and \ } \\ \label{brack}
[J_{(a, b)}, J_{(c, d)}] &= (u^{-bc} - u^{- ad})J_{(a + c, b + d)}
\end{align}
for $a, b, c, d \in \mathbb{Z}_p$. For $a, k \in \mathbb{Z}_p$ with $a \neq 0$, $J_{(a, ka)}$ and $J_{(0, a)}$ are elements of $ \mathfrak{sl}_p(\mathbb{F})$ by (\ref{trzero}). For a fixed $k \in \mathbb{Z}_p$, it follows immediately from the definitions of $D$ and $P$ that $J_{(1,k)}, J_{(2,2k)}, \ldots, J_{(p - 1,k(p - 1))}$ are linearly independent. Construct the following subalgebras:
\begin{align*}
H_k &= \langle J_{(a, ka)} \mid a \in \mathbb{Z}_p^\times \rangle_\mathbb{F}, k \in \mathbb{Z}_p \text{ \ and } \\
 H_\infty &= \langle J_{(0, a)} \mid a \in \mathbb{Z}_p^\times \rangle_\mathbb{F} = \langle P, P^2, \ldots, P^{p - 1} \rangle_\mathbb{F}.
 \end{align*}
 %By  (\ref{brack}), $H_\infty$ and $H_k$ are Lie subalgebras of $\mathfrak{sl}_p(R)$. 

Let 
\[
X = 
\begin{bmatrix}
1 & u^{\frac{p(p - 1)}{2}} & u^{\frac{(p - 1)(p - 2)}{2}}& \cdots & u^3 & u \\
u & 1 &  u^{\frac{p(p - 1)}{2}} & \cdots & u^6 & u^3 \\
u^3 & u & 1 & \cdots & u^{10} & u^6 \\
\vdots & \vdots  & \vdots &   \ddots & \vdots & \vdots \\
u^{\frac{(p - 1)(p - 2)}{2}} & u^{\frac{(p - 2)(p - 3)}{2}} & u^{\frac{(p - 3)(p - 4)}{2}} & \cdots & 1 & u^{\frac{p(p - 1)}{2}}\\
u^{\frac{p(p - 1)}{2}} &  u^{\frac{(p - 1)(p - 2)}{2}} & u^{\frac{(p - 2)(p - 3)}{2}} & \cdots & u & 1
\end{bmatrix}.
\] 
Since $p > 2$, $X$ is invertible. It is straightforward to verify that $X^{-1}DPX = D$ and $X^{-1} P X = P$. Thus by (\ref{commuPD}), conjugation by the matrix $X$ shifts $H_0, H_1, \ldots, H_{p-1}$ cyclically and fixes $H_\infty$. Moreover, $P$ is diagonalizable with eigenvalues listed in $D$. Thus, $H_0$ is conjugate to $H_\infty$. Note that $H_0$ is clearly a classical Cartan subalgebra. These imply that all $H_k, k \in \mathbb{Z}_p$ and $H_\infty$ are classical Cartan subalgebras of $\mathfrak{sl}_p(\mathbb{F})$. Next, we show that 
\begin{align}\label{direct}
\mathfrak{sl}_p(\mathbb{F}) = H_\infty \oplus H_0 \oplus H_1 \oplus \ldots \oplus H_{p - 1},
\end{align}
as a vector space. It is clear from the construction that $H_0 \cap \sum_{j \neq 0} H_j = \{ 0 \}$. In particular, the sum is direct for $H_0$ and $H_\infty$. Thus, the sums for all $H_i$'s are also direct, and so by counting dimensions, we have the equality.

We prove that the decomposition (\ref{direct}) is pairwise orthogonal with respect to the Killing form $K(A, B) = 2p \tr(AB)$.  It is obvious that $H_\infty$ is orthogonal to all the others $H_i$'s. Let $a, b \in \mathbb{Z}_p^\times$ and $k_1, k_2 \in \mathbb{Z}_p$ with $k_1 \neq k_2$. Then 
$(a + b, k_1a + k_2b) \neq (0, 0)$
and so by (\ref{brac}), 
\begin{align*}
K(J_{(a, k_1a)}, J_{(b, k_2b)}) &= 2p \tr (J_{(a, k_1a)} J_{(b, k_2b)}) \\
& = 2p u^{-k_1ab} \tr(J_{(a + b, k_1a + k_2 b)}) \\
& = 0.
\end{align*}
Thus, $H_i$ and $H_j$ are orthogonal for all $i, j \in \mathbb{Z}_p$ and $i \neq j$. 
This completes the proof for the case $m = 1$.

Next suppose that $m \geq 2$. Let $\mathbb{F}_{p^m}$ be the finite field of $p^m$ elements and $W = \mathbb{F}_{p^m} \oplus \mathbb{F}_{p^m}$ a $2m$-dimensional vector space over $\mathbb{Z}_{p}$ equipped with a symplectic form $ \langle \cdot , \cdot \rangle : W \times W \rightarrow \mathbb{Z}_p$ defined by the field trace \footnote[1]{The field trace of $\alpha \in \mathbb{F}_{p^m}$ is defined to be the sum of all Galois conjugates of $\alpha$, i.e. 
	\[
	Tr _{\mathbb{F}_{p^m}/\mathbb{Z}_p}(\alpha) = \alpha + \alpha^p + \cdots+ \alpha^{p^{m - 1}}.
	\]
} 
	 as follows: for any elements $\vec{w} = (\alpha; \beta), \vec{w}' = (\alpha'; \beta') \in W$,
\[
\langle \vec{w}, \vec{w}'  \rangle := Tr _{\mathbb{F}_{p^m}/\mathbb{Z}_{p}}(\alpha \beta' - \alpha' \beta).
\]
Then, by Corollary 3.3 of \cite{ZW02}, $W$ possesses a symplectic basis 
$$
\mathcal{B} = \{ \vec{e}_1, \ldots, \vec{e}_m, \vec{f}_1,\ldots, \vec{f}_m \}
$$ 
where $\{\vec{e}_1, \ldots, \vec{e}_m \}$ and $\{\vec{f}_1,\ldots, \vec{f}_m \}$ span the first and the second factors, respectively, such that

\[
\langle \vec{w}, \vec{w}' \rangle = \sum_{i = 1}^{m}(a_ib'_i - a'_ib_i),
\]
where $\vec{w} = \sum_{i = 1}^{m}(a_i\vec{e}_i + b_i\vec{f}_i)$ and $\vec{w}' = \sum_{i = 1}^{m}(a'_i\vec{e}_i + b'_i\vec{f}_i)$. With the basis $\mathcal{B}$, write each vector $\vec{w} \in W$ as
\[
\vec{w} = (a_1, \ldots, a_m; b_1, \ldots, b_m),
\]
and associate it with a matrix
\[
\mathcal{J}_{\vec{w}} = J_{(a_1, b_1)} \otimes J_{(a_2, b_2)} \otimes \cdots \otimes J_{(a_m, b_m)},
\]
where $\otimes$ denotes the Kronecker product of matrices \footnote[2]{The Kronecker product of an $m \times n$ matrix $A = (a_{ij})$ and a $p \times q$ matrix $B$ is defined to be the $mp \times nq$ block matrix:
$
 A \otimes B = 
 \begin{bmatrix}
 a_{11}B & \cdots & a_{1n}B \\
  \vdots & \ddots &  \vdots \\
 a_{m1}B & \cdots & a_{mn}B
 \end{bmatrix}.
$}, and $J_{(a_i, b_i)}$ is given as in the case $m = 1$ with a given primitive $p$th root of unity $u \in \mathbb{F}^\times$ for all $i = 1, 2, \ldots, m$. Then the set $\{ \mathcal{J}_{\vec{w}} \mid 0 \neq \vec{w} \in W \}$ forms a basis of $\mathfrak{sl}_{p^m}(\mathbb{F})$. By properties of the Kronecker product, we have the following identities:
\begin{align}\label{brac1}
\mathcal{J}_{\vec{w}}\mathcal{J}_{\vec{w}'} &= u^{-\mathfrak{B}(\vec{w}, \vec{w}')}\mathcal{J}_{\vec{w} + \vec{w}'} \text{ \ and \ } \\ \label{brac1.5}
 [\mathcal{J}_{\vec{w}}, \mathcal{J}_{\vec{w}'}] &= (u^{-\mathfrak{B}(\vec{w}, \vec{w}')} -u^{-\mathfrak{B}(\vec{w}', \vec{w})})\mathcal{J}_{\vec{w} + \vec{w}'} \\
 &= u^{-\mathfrak{B}(\vec{w}', \vec{w})}(u^{\langle \vec{w}, \vec{w}' \rangle} - 1) \mathcal{J}_{\vec{w} + \vec{w}'}, \nonumber
\end{align} 
where 
\[
\mathfrak{B}(\vec{w}, \vec{w}') = \sum_{i = 1}^m a'_i b_i
\]
for all $\vec{w} = (a_1, \ldots, a_m; b_1, \ldots, b_m), \vec{w}' = (a'_1, \ldots, a'_m; b'_1, \ldots, b'_m) \in W$. 

Write $\vec{w} = (\alpha; \beta) \in W$, where $\alpha = (a_1, a_2, \ldots, a_m)$ and $\beta = (b_1, b_2, \ldots, b_m)$. Define 
\begin{align*}
 H_\infty =  \langle \mathcal{J}_{(0; \lambda )} \mid \lambda \in \mathbb{F}^\times_q \rangle_{\mathbb{F}} \text{ \ and \ } 
H_\alpha = \langle \mathcal{J}_{( \lambda; \alpha \lambda )} \mid \lambda \in \mathbb{F}^\times_q \rangle_{\mathbb{F}}, 
\end{align*}
 where $\alpha \in \mathbb{F}_{p^m}$. Since all $\mathcal{J}_{\vec{w}}$'s are basis elements, we have  
 \begin{align}\label{direct2}
 \mathfrak{sl}_q({\mathbb{F}}) = H_\infty \oplus (\oplus_{\alpha \in \mathbb{F}_{p^m}} H_\alpha).
 \end{align}
 
 Using the similar argument to the case where $m = 1$ and the fact that an eigenvalue of the Kronecker product is the product of each eigenvalue from each component with the corresponding eigenvector formed by the tensor product of each, we see that all $H_i's$ are classical Cartan subalgebras.  
 
 To see that they are pairwise orthogonal, note that if $(\gamma; \delta) \neq (-\alpha; -\beta)$, then $\tr (\mathcal{J}_{(\alpha; \beta)}\mathcal{J}_{(\gamma; \delta)}) = 0$. Indeed, if $\lambda = (a_1, \ldots, a_m), \beta = (b_1, \ldots, b_m), \gamma = (a'_1, \ldots, a'_m), \\ \delta = (b'_1, \ldots, b'_m)$ and $a_i \neq -a'_i$ for some $i \in \{1, \ldots, m\}$, then $a_i + a'_i \neq 0$ and $\tr (J_{(a_i + a'_i, b_i + b'_i)}) = 0$ (as in the case $m = 1$). By (\ref{brac1}) and the trace property of the Kronecker product,
 \begin{align*} 
 \tr (\mathcal{J}_{(\alpha; \beta)}\mathcal{J}_{(\gamma; \delta)}) &= u^{-\mathfrak{B}((\alpha;\beta), (\gamma; \delta))} \tr (\mathcal{J}_{(a_1+a_1', \ldots, a_m + a_m' ; b_1 + b_1', \ldots, b_m + b_m')}) \\
  &= u^{-\mathfrak{B}((\alpha;\beta), (\gamma; \delta))} \tr(\otimes_{j = 1}^m J_{(a_j + a_j', b_j + b_j')}) \\
  &= u^{-\mathfrak{B}((\alpha;\beta), (\gamma; \delta))} \prod_{j=1}^{m}\tr( J_{(a_j + a_j', b_j + b_j')}) \\
  & = 0.
   \end{align*}
  This completes the proof.
\end{proof}

The following corollary is immediate.

\begin{cor} \label{coralgeclosed}
	Let $\mathbb{F}$ be an algebraically closed field of positive characteristic and let $n = p^m$ be a prime power. If $\cha(\mathbb{F}) \neq 2, 3$ and $p$, then $\mathfrak{sl}_n(\mathbb{F})$ has a COD.
\end{cor}

Let $\mathbb{F}_q$ denote the finite field with $q$ elements. It is known that $-1$ is a square in $\mathbb{F}_q$ if and only if $q$ is even or $q \equiv 1 \pmod{4}$ and for a prime $p$, $\mathbb{F}_q$ has a primitive $p$th root of unity if and only if $p \mid (q - 1)$, by Cauchy theorem. Thus, we have:

\begin{cor}
\label{corfinitefield}
	Let $n = p^m$ be a prime power. Assume that $\cha(\mathbb{F}_q) \neq 2, 3$ and $p$. If
	\begin{enumerate}
	    \item  $p = 2$ and $q \equiv 1 \pmod{4}$, or
	    \item $p > 2$ and  $p \mid (q - 1)$,
	\end{enumerate}
	then 
	 $\mathfrak{sl}_n(\mathbb{F}_q)$ has a COD.
\end{cor}

\section{Characterization of fields for COD of $\mathfrak{sl}_n, n = 2, 3$} \label{characterization}

Let $\mathbb{F}$ be a field of positive characteristic $>3$.
Assume that $\mathfrak{sl}_n(\mathbb{F})$ is classical. Then all its classical Cartan subalgebras are conjugate \cite{S67}. We say that any two CODs of $\mathfrak{sl}_n(\mathbb{F})$ are {\it conjugate} if there exists an isomorphism in $\Aut (\mathfrak{sl}_n(\mathbb{F}))$ mapping each component from the first decomposition to exactly one component in the second. 
Let $H_0$ be the classical Cartan subalgebra of $\mathfrak{sl}_n{(\mathbb{F})}$ consisting of the diagonal matrices. Note that the conjugation preserves the orthogonality with respect to the Killing form $K$. We may assume that the first component of a COD of $\mathfrak{sl}_n(\mathbb{F})$ is $H_0$. Let $H$ be another classical Cartan subalgebra of $\mathfrak{sl}_n{(\mathbb{F})}$ orthogonal to $H_0$ with respect to $K$. According to the corollary to Lemma II.1.2 of \cite{S67}, $K$ is nondegenerate, and so is its restriction to $H_0$.
Since $H$ and $H_0$ are conjugate, $K_{|_H}$ is also non-degenerate. From these properties, we have the following lemma, but we omit the proof since the similar arguments in \cite [Page 32]{KT94} and \cite[Lemma 2.1]{SZ20} can be applied. 

\begin{lemma}\label{H}
	Every classical Cartan subalgebra orthogonal to $H_0$ has a basis of the form indicated below.
	\begin{enumerate}[(1)]
		\item If $n = 2$, then 
		\[
		H = \bigg \langle 
		\begin{pmatrix}
		0 & 1 \\
		a & 0
		\end{pmatrix}
		\bigg \rangle_{\mathbb{F}} \]
		for some $a \in \mathbb{F} \setminus \{0\}$.
		\item If $n = 3$, then 
		\[
		H = \Bigg \langle 
		\begin{pmatrix}
		0 & 1 & 0\\
		0 & 0 & a \\
		ab & 0 & 0
		\end{pmatrix},
		\begin{pmatrix}
		0 & 0 & 1\\
		ab & 0 & 0 \\
		0 & b & 0
		\end{pmatrix}
		\Bigg \rangle_{\mathbb{F}} \]
		for some $a, b \in \mathbb{F} \setminus \{0\}$.
	\end{enumerate}	
\end{lemma}

Next, we characterize fields $\mathbb{F}$ allowing $\mathfrak{sl}_n, n =2, 3$ to admit a COD. Note that only the necessary condition of the following theorem needs to be proved. 

\begin{thm} \label{main2}
Let $\mathbb{F}$ be a field of characteristic $> 3$. Then $\mathfrak{sl}_2(\mathbb{F})$ has a unique (up to conjugacy) COD if and only if $\mathbb{F}$ contains a primitive fourth root of unity. 
\end{thm}
\begin{proof}
Assume that $\mathfrak{sl}_2(\mathbb{F})$ admits a COD. Then by Lemma \ref{H}, the COD must be of the form 
\[
\mathfrak{sl}_2(\mathbb{F}) =\bigg \langle 
\begin{pmatrix}
1 & 0 \\
0 & -1
\end{pmatrix}
\bigg \rangle_{\mathbb{F}}
\oplus
\bigg \langle 
\begin{pmatrix}
0 & 1 \\
a & 0
\end{pmatrix}
\bigg \rangle_{\mathbb{F}}
\oplus
\bigg \langle 
\begin{pmatrix}
0 & 1 \\
b & 0
\end{pmatrix}
\bigg \rangle_{\mathbb{F}}
\]
for some $a, b \neq 0$. The orthogonality implies that $b = -a$. Since all components are conjugate, the middle one must be diagonalizable and $a$ is forced to be a square in $\mathbb{F}$. The last component is classical if and only if $-a$ is a square, which implies $\sqrt{-1} \in \mathbb{F}$.

The uniqueness follows easily from the fact that $\bigg \langle 
\begin{pmatrix}
0 & 1 \\
a & 0
\end{pmatrix}
\bigg \rangle_{\mathbb{F}}$ is conjugate to $\bigg \langle 
\begin{pmatrix}
0 & 1 \\
1 & 0
\end{pmatrix}
\bigg \rangle_{\mathbb{F}}$ by the matrix
$
\begin{pmatrix}
\sqrt{a} & 0 \\
0 & 1
\end{pmatrix}
$.
\end{proof}

For the finite field $\mathbb{F}_q$, we have the following characterization. 

\begin{cor}\label{corfinite2}
 Assume that $\cha (\mathbb{F}_q) > 3$. Then $\mathfrak{sl}_2(\mathbb{F}_q)$ has a unique (up to conjugacy) COD if and only if $q \equiv 1 \pmod{4}$.    
\end{cor}

\begin{rmk}
If $q \equiv 3 \pmod{4}$, then $\mathfrak{sl}_2(\mathbb{F}_q)$ has only two classical orthogonal components which are 
$
\bigg \langle 
\begin{pmatrix}
1 & 0 \\
0 & -1
\end{pmatrix}
\bigg \rangle_{\mathbb{F}_q}
$
and
$
\bigg
\langle 
\begin{pmatrix}
0 & 1 \\
1 & 0
\end{pmatrix}
\bigg \rangle_{\mathbb{F}_q}
$.
\end{rmk}

\begin{thm} \label{main3}
Let $\mathbb{F}$ be a field of characteristic $> 3$. Then $\mathfrak{sl}_3(\mathbb{F})$ has a COD if and only if $\mathbb{F}$ contains a primitive cube root of unity.
\end{thm}
\begin{proof}
Suppose that $\mathbb{F}$ does not have any primitive cube root of unity but $\mathfrak{sl}_3(\mathbb{F})$ possesses a COD. Note that the decomposition of $\mathfrak{sl}_3(\mathbb{F})$ has four components. Then, up to conjugacy, we can assume that the classical Cartan subalgebra $H_0$ consisting of diagonal matrices is the first component and, by Lemma \ref{H}, all other components are of the forms 
\begin{align*}
H'_1 &= \Bigg \langle 
\begin{pmatrix}
0 & 1 & 0\\
0 & 0 & a \\
ab & 0 & 0
\end{pmatrix},
\begin{pmatrix}
0 & 0 & 1\\
ab & 0 & 0 \\
0 & b & 0
\end{pmatrix}
\Bigg \rangle_{\mathbb{F}}, \\
H'_2 &= \Bigg \langle 
\begin{pmatrix}
0 & 1 & 0\\
0 & 0 & c \\
cd & 0 & 0
\end{pmatrix},
\begin{pmatrix}
0 & 0 & 1\\
cd & 0 & 0 \\
0 & d & 0
\end{pmatrix}
\Bigg \rangle_{\mathbb{F}}, \\
H'_3 &= \Bigg \langle 
\begin{pmatrix}
0 & 1 & 0\\
0 & 0 & e \\
ef & 0 & 0
\end{pmatrix},
\begin{pmatrix}
0 & 0 & 1\\
ef & 0 & 0 \\
0 & f & 0
\end{pmatrix}
\Bigg \rangle_{\mathbb{F}}
\end{align*}
for some $a, b, c, d, e, f \neq 0$. By the orthogonality between $H'_1$ and $H'_2$, we have 
$
cd + ad + ab = 0 \text{ and } cd + cb + ab = 0.
$
Then
$d = a^{-1}cb$. Substituting $d$ in the first equation, we get
$
c^2 + a c + a^2 = 0.
$
Since there is no primitive cube root of unity in $\mathbb{F}$, the polynomial $t^2 + a t + a^2$ has no root in $\mathbb{F}$. This is a contradiction.
\end{proof}

We immediately have the following corollary.

\begin{cor}\label{corfinite3}
 Assume that $\cha (\mathbb{F}_q) > 3$. Then $\mathfrak{sl}_3(\mathbb{F}_q)$ has a COD if and only if $3 \mid (q - 1)$.    
\end{cor}

\begin{rmk}
If $3 \nmid (q - 1)$, then the absence of primitive cube roots of unity implies that $\mathfrak{sl}_3(\mathbb{F}_q)$ does not possess an orthogonal pair of classical Cartan subalgebras.
\end{rmk}

Note that if $\mathbb{F}$ is algebraically closed, the COD of $\mathfrak{sl}_3(\mathbb{F})$ is immediately unique (up to conjugacy). 
The uniqueness result is the same for $\mathfrak{sl}_3$ over a finite field $\mathbb{F}_q$ with $3 \mid (q - 1)$, but the arguments of proof are different from the algebraically closed field case. 
We provide the precise details in the following discussion. 

\begin{thm} \label{main3alclosed}
Let $\mathbb{F}$ be an algebraically closed field of characteristic $> 3$. Then the COD of 
$\mathfrak{sl}_3(\mathbb{F})$ is unique (up to conjugacy).
\end{thm}
\begin{proof}
By the above discussion, we can assume that one component of a COD is the subalgebra $H_0$ consisting of the diagonal matrices. According to Lemma \ref{H}, we have a COD 
\[
\mathfrak{sl}_n(\mathbb{F}) = H_0 \oplus H_1 \oplus H_2 \oplus H_3,
\]
where
\begin{align*}
H_1 &= \Bigg \langle 
\begin{pmatrix}
0 & 1 & 0\\
0 & 0 & 1 \\
1 & 0 & 0
\end{pmatrix},
\begin{pmatrix}
0 & 0 & 1\\
1 & 0 & 0 \\
0 & 1 & 0
\end{pmatrix}
\Bigg \rangle_{\mathbb{F}}, \\
H_2 &= \Bigg \langle 
\begin{pmatrix}
0 & 1 & 0\\
0 & 0 & a \\
ab & 0 & 0
\end{pmatrix},
\begin{pmatrix}
0 & 0 & 1\\
ab & 0 & 0 \\
0 & b & 0
\end{pmatrix}
\Bigg \rangle_{\mathbb{F}}, \\
H_3 &= \Bigg \langle 
\begin{pmatrix}
0 & 1 & 0\\
0 & 0 & c \\
cd & 0 & 0
\end{pmatrix},
\begin{pmatrix}
0 & 0 & 1\\
cd & 0 & 0 \\
0 & d & 0
\end{pmatrix}
\Bigg \rangle_{\mathbb{F}}
\end{align*}
for some $a, b, c, d \neq 0$, because $\mathbb{F}$ is algebraically closed. Using the similar arguments and equations in the proof of Theorem \ref{main3}, we have $a = b, c = d \in \{u, u^2\}$, where $u$ is a primitive cube root of unity. The orthogonality between $H_2$ and $H_3$ allows us to assume that $a = u$ and $c = u^2$. Therefore, the COD is uniquely defined.
\end{proof}

Now, we consider $\mathfrak{sl}_3(\mathbb{F}_q)$ with $\cha (\mathbb{F}_q) > 3$. Due to the nonexistence of cube roots for some elements in $\mathbb{F}_q$, this case becomes more intricate.

In general, if $3 \mid (q - 1)$, then a COD of $\mathfrak{sl}_3(\mathbb{F}_q)$ is of the form 
\[
\mathfrak{sl}_3(\mathbb{F}_q) = H_0 \oplus H_1 \oplus H_2 \oplus H_3
\]
where
\begin{align*}
H_1 &= \Bigg \langle 
\begin{pmatrix}
0 & 1 & 0\\
0 & 0 & a \\
ab & 0 & 0
\end{pmatrix},
\begin{pmatrix}
0 & 0 & 1\\
ab & 0 & 0 \\
0 & b & 0
\end{pmatrix}
\Bigg \rangle_{\mathbb{F}_q}, \\
H_2 &= \Bigg \langle 
\begin{pmatrix}
0 & 1 & 0\\
0 & 0 & ua \\
u^2ab & 0 & 0
\end{pmatrix},
\begin{pmatrix}
0 & 0 & 1\\
u^2ab & 0 & 0 \\
0 & ub & 0
\end{pmatrix}
\Bigg \rangle_{\mathbb{F}_q}, \\
H_3 &= \Bigg \langle 
\begin{pmatrix}
0 & 1 & 0\\
0 & 0 & u^2a \\
uab & 0 & 0
\end{pmatrix},
\begin{pmatrix}
0 & 0 & 1\\
uab & 0 & 0 \\
0 & u^2b & 0
\end{pmatrix}
\Bigg \rangle_{\mathbb{F}_q}
\end{align*}
for some $a, b \neq 0$ and a primitive cube root of unity $u$. This can be verified by the similar arguments as above. Thus, motivated by \cite{KT94}, we denote this decomposition by $J_3(a, b)$ (sometimes we may call it $J_3$-decomposition). The reader may check the terminology $J$-decomposition in this reference for clarification.

We investigate how many $J_3(a, b)$ that $\mathfrak{sl}_3(\mathbb{F}_q)$ can have (up to conjugacy).
Assume that $3 \mid (q - 1)$. Then $q - 1 = 3 l$ for some $l \in \mathbb{Z}_{>0}$. Let $x$ be a generator of $\mathbb{F}_q^\times$.
Then $|\braket{x^3}| = l$ and the index $[\mathbb{F}_q^\times: \braket{x^3}] = 3$. Let $z \in \mathbb{F}_q^\times \setminus \braket{x^3}$. We focus on these three distinct cosets: $\braket{x^3}, z\braket{x^3}$ and $z^2\braket{x^3}$. We use them to find the number of CODs of $\mathfrak{sl}_3(\mathbb{F}_q)$.

Our goal is to show that there are exactly two $J_3$-decompositions of $\mathfrak{sl}_3(\mathbb{F}_q)$, where $3 \mid (q - 1)$. Then we will see that only one $J_3$-decomposition for this Lie algebra is a COD.  
For convenience, we define a relation 
\[
J_3(a, b) \approx J_3(c, d) \Longleftrightarrow J_3(a, b) \text{ is conjugate to } J_3(c, d)
\]
for all $ a, b, c, d \neq 0$. 
Obviously, this is an equivalence relation.

The following lemma provides a sufficient condition to have two $J_3$-decompositions conjugate.

\begin{lemma}\label{appro}
	Suppose that $3 \mid (q - 1)$. For nonzero $a, b, c, d \in \mathbb{F}_q$, if $a^{-1} c$ and $b^{-1} d$ are in the same left coset defined by $\braket{x^3}$, then $J_3(a, b) \approx J_3(c, d)$.
\end{lemma}
\begin{proof}
	Assume that $a^{-1} c$ and $b^{-1} d$ are in the same left coset. Then $(a^{-1} c)^2 b^{-1} d$ and $a^{-1} c (b^{-1} d)^2$ have cube roots. Thus, we can use the matrix conjugation defined by
	\[
	\begin{pmatrix}
	1 &   & \\
	& \sqrt[3]{(a^{-1} c)^2 b^{-1} d} & \\
	&   &  \sqrt[3]{a^{-1} c (b^{-1} d)^2}
	\end{pmatrix}
	\]
	to map $J_3(a, b)$ to $J_3(c, d)$ isomorphically.
\end{proof}

 From the above lemma, it follows that if $a$ and $b$ are in the same coset defined by $\braket{x^3}$, then $J_3(a, b) \approx J_3(1, 1)$. On the other hand, if they are in different cosets, we will show that $J_3(a, b)$ is conjugate to either $J_3(1, z)$ or $J_3(1, z^2)$. As a result, we have at most three $J_3$-decompositions of $\mathfrak{sl}_3(\mathbb{F}_q)$ up to conjugacy, when $3 \mid (q - 1)$.

\begin{prop}\label{ab1z}
	Under the above setting, the following statements hold.
	\begin{enumerate}[(1)]
		\item If $a$ and $b$ are in the same coset, then $J_3(a, b) \approx J_3(1, 1)$.
		\item If $a$ and $b$ are in different cosets, then $J_3(a, b) \approx J_3(1, z)$ or $J_3(a, b) \approx J_3(1, z^2)$.
	\end{enumerate}
\end{prop}
\begin{proof}
We only need to prove (2).
Consider the isomorphism $\varphi : \mathbb{F}_q^\times / \braket{x^3} \rightarrow (\mathbb{Z}_3, +)$ that sends $z^i\braket{x^3}$ to $i$.
 For $i = 0, 1, 2$, let $\overline{z}^i$ denote the coset $z^i\braket{x^3}$. Now, assume that $a$ and $b$ are in two distinct cosets, say $a \in \overline{z}^i$ and $b \in \overline{z}^j$, we write $(a, b) \in (\overline{z}^i, \overline{z}^j)$. Thus, if $(c, d) \in (\overline{z}^s, \overline{z}^t)$, then $(a + c, b + d) \in (\overline{z}^{i + s}, \overline{z}^{j + t})$. Note that $(1, z) \in (\overline{z}^0, \overline{z}^1)$ and so $(1, z^{-1}) \in (\overline{z}^0, \overline{z}^2)$.
	By Lemma \ref{appro}, for any $(a, b)$ in $(\overline{z}^0, \overline{z}^1), (\overline{z}^2, \overline{z}^0)$ and $(\overline{z}^1, \overline{z}^2)$, we have $J_3(a, b) \approx J_3(1, z)$. Finally, $(1, z^2) \in (\overline{z}^0, \overline{z}^2)$ and so $( 1, z^{-2}) \in (\overline{z}^0, \overline{z}^1)$. By Lemma \ref{appro} again, for any $(a, b)$ in $(\overline{z}^1, \overline{z}^0), (\overline{z}^0, \overline{z}^2)$ and $(\overline{z}^2, \overline{z}^1)$, we have $J_3(a, b) \approx J_3(1, z^2)$.
\end{proof}

To conclude that there are exactly two $J_3$-decompositions of $\mathfrak{sl}_3(\mathbb{F}_q)$ up to conjugacy, when $3 \mid (q - 1)$, we prove that $J_3(1, z)$ is conjugate to $J_3(1, z^2)$ but not to $J_3(1, 1)$.

\begin{thm}\label{numberofthree}
Assume that $\cha (\mathbb{F}_q) > 3$. If $3 \mid (q - 1)$, then $\mathfrak{sl}_3(\mathbb{F}_q)$ has exactly two $J_3$-decompositions up to conjugacy, which are represented by $J_3(1, 1)$ and $J_3(1, z)$.
\end{thm}
\begin{proof}
	The 
	 $J_3(1, 1)$ of $\mathfrak{sl}_3{}(\mathbb{F}_q)$ is 
	\[
	\mathfrak{sl}_3(\mathbb{F}_q) = \braket{J_{(1, 0)}, J_{(2, 0)}} \oplus \braket{J_{(0, 1)}, J_{(0, 2)}} \oplus \braket{J_{(1, 1)}, J_{(2, 2)}} \oplus \braket{J_{(2, 1)}, J_{(1, 2)}},
	\]
	where $J_{(a, b)} = D^a P^b$ and 
	\[
	D= \diag (1, u, u^2),
	P=
	\begin{pmatrix}
	0 & 0 & 1 \\
	1 & 0 & 0 \\
	0 & 1 & 0
	\end{pmatrix}.
	\]
	Note that 	$[J_{(a, b)}, J_{(c, d)}] = (u^{-bc} - u^{- ad})J_{(a + c, b + d)}$ for all $a, b \in \{0, 1, 2\}$ (cf. eq. (\ref{brack})). 
	
	The following description is for $J_3(1, z)$. Let 
	\[
	P_0 = I_3,
	P_1 =
	\begin{pmatrix}
	0 & 0 & 1 \\
	z & 0 & 0 \\
	0 & z & 0 
	\end{pmatrix}
	\text{ and }
	P_2 =
	\begin{pmatrix}
	0 & 1 & 0 \\
	0 & 0 & 1 \\
	z & 0 & 0 
	\end{pmatrix}.
	\]
	Define $J'_{(a, b)} = D^a P_b$ for all $a, b \in \{0, 1, 2\}$. Let $m_{bd} = \min\{b, d\} \pmod{2}$. Then 
	\begin{align}\label{J'}
	[J'_{(a, b)}, J'_{(c, d)}] = z^{m_{bd}}(u^{-bc} - u^{-ad})J'_{(a + c, b + d)}.
	\end{align}
	The $J_3(1, z)$ of $\mathfrak{sl}_3{}(\mathbb{F}_q)$ is 
	\[
	\mathfrak{sl}_3(\mathbb{F}_q) = \braket{J'_{(1, 0)}, J'_{(2, 0)}} \oplus \braket{J'_{(0, 1)}, J'_{(0, 2)}} \oplus \braket{J'_{(1, 1)}, J'_{(2, 2)}} \oplus \braket{J'_{(2, 1)}, J'_{(1, 2)}}.
	\]
	Suppose that $J_3(1, z) \approx J_3(1, 1)$ by an automorphism $\varphi$. We show that this leads to a contradiction. Since $\varphi$ sends one component of $J_3(1, z)$ to exactly one component of $J_3(1, 1)$, for each $(a, b)$, by (\ref{J'}) there exists a unique $(a', b')$ such that $\varphi(J'_{(a, b)}) = \alpha_{a, b}J_{(a', b')}$ where $\alpha_{a, b} \in \mathbb{F}_q^\times$. We consider all possible cases. 
     \\
	{\bf Case 1}
	\begin{align*}
	\varphi : &J'_{(1, 0)} \mapsto a J_{(1, 0)}, J'_{(0, 1)} \mapsto c J_{(0, 1)} \\ &J'_{(2, 0)} \mapsto b J_{(2, 0)}, J'_{(0, 2)} \mapsto d J_{(0, 2)}  
	\end{align*}
	for some $a, b, c, d \in \mathbb{F}_q^\times$. Using (\ref{J'}), we obtain 
	\begin{enumerate}[(i)]
		\item $\varphi (J'_{(1, 1)}) = a c J_{(1, 1)}$,
		\item $\varphi (J'_{(2, 2)}) = b d J_{(2, 2)}$,
		\item $\varphi (J'_{(1, 2)}) = a d J_{(1, 2)}$.
	\end{enumerate}
	Since $$\varphi([J'_{(0, 1)}, J'_{(1, 1)}]) = [\varphi(J'_{(0, 1)}), \varphi(J'_{(1, 1)})]$$ and  $$\varphi([J'_{(0, 1)}, J'_{(2, 2)}]) = [\varphi(J'_{(0, 1)}), \varphi(J'_{(2, 2)})],$$ by (i), (ii) and (iii), we have $z d = c^2$ and $z = c d$, accordingly. This forces $z = d^3$ which contradicts the choice of $z$.\\
	{\bf Case 2}
	\begin{align*}
	\varphi : &J'_{(2, 0)} \mapsto a J_{(1, 0)}, J'_{(0, 1)} \mapsto c J_{(0, 1)} \\ &J'_{(1, 0)} \mapsto b J_{(2, 0)}, J'_{(0, 2)} \mapsto d J_{(0, 2)}  
	\end{align*}
	for some $a, b, c, d \in \mathbb{F}_q^\times$. Using (\ref{J'}), we obtain 
	\begin{enumerate}[(i)]
		\item $\varphi (J'_{(1, 1)}) = (\frac{b c}{1 + u}) J_{(2, 1)}$,
		\item $\varphi (J'_{(2, 2)}) = (\frac{a d}{1 + u}) J_{(1, 2)}$,
		\item $\varphi (J'_{(1, 2)}) =  b d(1 + u) J_{(2, 2)}$.
	\end{enumerate}
	Since $$\varphi([J'_{(0, 1)}, J'_{(1, 1)}]) = [\varphi(J'_{(0, 1)}), \varphi(J'_{(1, 1)})]$$ and  $$\varphi([J'_{(0, 1)}, J'_{(2, 2)}]) = [\varphi(J'_{(0, 1)}), \varphi(J'_{(2, 2)})],$$ by (i), (ii) and (iii), we have $z d (1 + u)^3 = c^2$ and $z = c d$, accordingly. This forces $z = d^3(1 + u)^3$ which contradicts the choice of $z$.
	
    Here, we have provided the details for only two cases.  For the remaining cases, we refer the reader to Appendix \ref{Mathematicacode} for a Mathematica code to verify that $z$ would be a cube unit in $\mathbb{F}_q$, thus reaching a contradiction. 
		
	Finally, let 
	\[
	Q_0 = I_3,
	Q_1 =
	\begin{pmatrix}
	0 & 0 & 1 \\
	z^2 & 0 & 0 \\
	0 & z^2 & 0 
	\end{pmatrix}
	\text{ and }
	Q_2 =
	\begin{pmatrix}
	0 & 1 & 0 \\
	0 & 0 & 1 \\
	z^2 & 0 & 0 
	\end{pmatrix}.
	\]
	Define $J''_{(a, b)} = D^a Q_b$ for all $a, b \in \{0, 1, 2\}$. Let $m_{bd} = \min\{b, d\} \pmod{2}$. Then 
	\begin{align}\label{J''}
	[J''_{(a, b)}, J''_{(c, d)}] = z^{2 m_{bd}}(u^{-bc} - u^{-ad})J''_{(a + c, b + d)}.
	\end{align}
	The $J_3(1, z^2)$ of $\mathfrak{sl}_3{}(\mathbb{F}_q)$ is 
	\[
	\mathfrak{sl}_3(\mathbb{F}_q) = \braket{J''_{(1, 0)}, J''_{(2, 0)}} \oplus \braket{J''_{(0, 1)}, J''_{(0, 2)}} \oplus \braket{J''_{(1, 1)}, J''_{(2, 2)}} \oplus \braket{J''_{(2, 1)}, J''_{(1, 2)}}.
	\]
	We find an automorphism for $\mathfrak{sl}_3(\mathbb{F}_q)$ that maps $J_3(1, z^2)$ to $J_3(1, z)$. To construct such an automorphism, we define a map $\psi$ on the basis of $J_3(1, z^2)$ as follows:
	\begin{align*}
	&J''(1, 0) \mapsto - J'(1, 0),  && J''(2, 0) \mapsto - J'(2, 0), \\ & J''(0, 1) \mapsto -z J'(0, 2), && J''(0, 2) \mapsto - J'(0, 1), \\ & J''(1, 1) \mapsto \frac{z}{1 + u} J'(1, 2), && J''(2, 2) \mapsto \frac{1}{1 + u} J'(2, 1) \\ & J''(1, 2) \mapsto (1 + u) J'(1, 1) && J''(2, 1) \mapsto z(1 + u)J'(2, 2).
	\end{align*}
	We extend $\psi$ linearly to the entire $J_3(1, z^2)$. Then	by using the fact that $u$ is a primitive cube root of unity together with  (\ref{J'}) and (\ref{J''}), we see that $\psi$ is an automorphism for $\mathfrak{sl}_3(\mathbb{F}_q)$. For the details of this part, we refer the reader to Appendix \ref{psi}. 
\end{proof}

Note that the component $H_1$ of $J_3(1, z)$ is not diagonalizable because $z$ is non-cube. Thus, $H_1$ is not conjugate to $H_0$, and so it is not classical. It implies that $J_3(1, z)$ is not a COD. We present the following theorem as a conclusion.

\begin{thm}\label{unique3}
    Assume that $\cha (\mathbb{F}_q) > 3$. If $3 \mid (q - 1)$, then $\mathfrak{sl}_3(\mathbb{F}_q)$ has a unique COD up to conjugacy represented by $J_3(1, 1)$.
\end{thm}

\section{Concluding Remark}
The standard OD for a simple Lie algebra over $\mathbb{C}$ only requires each component to be a Cartan subalgebra. As mentioned, for a modular Lie algebra, the Cartan subalgebras do not have the same properties as in the complex case. We thus proposed that a suitable generalization of OD in the modular case should be COD, as defined in this paper. The study of COD is more optimal, as classical Cartan subalgebras of a classical Lie algebra possess many desirable properties. Moreover, there are several ingredients that we can utilize to delve deeper into the intricacies of the problem.

The study of COD for $\mathfrak{sl}_n(\mathbb{F})$, when $n = 4, 5$, should be discussed in the next step. One may consider the COD problem for the case of $n = 6$. The results of this problem could have significant implications for the Winnie-the-Pooh problem over the complex numbers. Exploring the potential applications of the COD problem still requires further attention. We plan to study some of these topics in our future paper. 

\section*{Acknowledgment}
This work (Grant No. RGNS 64-096) was supported by Office of the Permanent Secretary, Ministry of Higher Education, Science, Research and Innovation (OPS MHESI), Thailand Science Research and Innovation (TSRI) and King Mongkut's University of Technology Thonburi.
The second author would like to thank Professor Yi Ming Zou for introducing the orthogonal decomposition problem and for his earlier suggestions.

\section{Appendix}
	\subsection{Mathematica code for checking $J_3(1, z)$ and $J_3(1, 1)$} \label{Mathematicacode}
	Recalling the setting of the proof of Theorem \ref{numberofthree},  we need to consider the following $J_3$-decompositions.
	\begin{enumerate}[(1)]
		\item $J_3(1, 1)$:
		\[
		\mathfrak{sl}_3(\mathbb{F}_q) = \braket{J_{(1, 0)}, J_{(2, 0)}} \oplus \braket{J_{(0, 1)}, J_{(0, 2)}} \oplus \braket{J_{(1, 1)}, J_{(2, 2)}} \oplus \braket{J_{(2, 1)}, J_{(1, 2)}},
		\]
		where $J_{(a, b)} = D^a P^b$ and 
		\[
		D= \diag (1, u, u^2), 
		P=
		\begin{pmatrix}
		0 & 0 & 1 \\
		1 & 0 & 0 \\
		0 & 1 & 0
		\end{pmatrix}.
		\]
		\item  $J_3(1, z)$:
		\[
		\mathfrak{sl}_3(\mathbb{F}_q) = \braket{J'_{(1, 0)}, J'_{(2, 0)}} \oplus \braket{J'_{(0, 1)}, J'_{(0, 2)}} \oplus \braket{J'_{(1, 1)}, J'_{(2, 2)}} \oplus \braket{J'_{(2, 1)}, J'_{(1, 2)}},
		\]
		where $J'_{(a, b)} = D^a P_b$ and 
		\[
		D= \diag (1, u, u^2), 
		P=
		\begin{pmatrix}
		0 & 0 & 1 \\
		1 & 0 & 0 \\
		0 & 1 & 0
		\end{pmatrix}.
		\] 
		\end{enumerate}
	The input of the above matrices is:
\begin{align*}
&P = \{\{0,0,1\},\{1,0,0\},\{0,1,0\}\}; \\ 
&p[0]=\text{IdentityMatrix}[3];\\
&p[1]=P;\\
&p[2]=P.P;\\
&p'[0]=\text{IdentityMatrix}[3];\\
&p'[1]=\{\{0,0,1\},\{z,0,0\},\{0,z,0\}\};\\
&p'[2]=\{\{0,1,0\},\{0,0,1\},\{z,0,0\}\};\\
&p''[0]=\text{IdentityMatrix}[3];\\
&p''[1]=\{\{0,0,1\},\{z^2,0,0\},\{0,z^2,0\}\};\\
&p''[2]=\{\{0,1,0\},\{0,0,1\},\{z^2,0,0\}\};\\
&d[0]=\text{IdentityMatrix}[3];\\
&d[i \_]\text{:=}\text{DiagonalMatrix}[\{1, u^{\text{Mod}[i,3]}, u^{\text{Mod}[2 i,3]} \}];\\
&f[A \_, B \_]\text{:=}A.B-B.A \\
&J[\text{a$\_$},\text{b$\_$}]\text{:=}d[\text{Mod}[a,3]].p[\text{Mod}[b,3]]\\
&J'[\text{a$\_$},\text{b$\_$}]\text{:=}d[\text{Mod}[a,3]].p'[\text{Mod}[b,3]]\\
&J''[\text{a$\_$},\text{b$\_$}]\text{:=}d[\text{Mod}[a,3]].p''[\text{Mod}[b,3]]
\end{align*}
Suppose that  $J_3(1, 1) \approx  J_3(1, z)$ by the map $\varphi$. Then we consider: 
\begin{align*}
&\varphi (J'[1,0])=\alpha  J[m[i],n[i]];\\
&\varphi (J'[2,0])=\beta  J[k[i],l[i]];\\
&\varphi (J'[0,1])=\gamma  J[s[i],t[i]];\\
&\varphi (J'[0,2])=\delta  J[x[i],y[i]];\\
&\varphi (J'[1,1])=\text{Simplify}\left[\frac{f\left[ \varphi (J'[1,0]),\varphi(J'[0,1])\right]}{1-u^2}\right];\\
&\varphi (J'[2,2])=\text{Simplify}\left[\frac{f\left[\varphi(J'[2,0]),\varphi(J'[0,2])\right]}{1-u^2}\right];\\
&\varphi (J'[1,2])=\text{Simplify}\left[\frac{f\left[\varphi(J'[1,0]),\varphi (J'[0,2])\right]}{1-u}\right];
\end{align*}
and check all of the following cases for $m[i], n[i], k[i], l[i], s[i], t[i], x[i]$ and $y[i]$.
\begin{align*}
&m [1] = 1; n [1] = 0; k[1] = 2; l[1] = 0; s[1] = 0; t[1] = 1; 
x[1] = 0; y[1] = 2; \\
&m [2] = 1; n [2] = 0; k[2] = 2; l[2] = 0; s[2] = 0; t[2] = 2; 
x[2] = 0; y[2] = 1; \\
&m [3] = 1; n [3] = 0; k[3] = 2; l[3] = 0; s[3] = 1; t[3] = 1; 
x[3] = 2; y[3] = 2;\\
&m [4] = 1; n [4] = 0; k[4] = 2; l[4] = 0; s[4] = 2; t[4] = 2; 
x[4] = 1; y[4] = 1;\\
&m [5] = 1; n [5] = 0; k[5] = 2; l[5] = 0; s[5] = 2; t[5] = 1; 
x[5] = 1; y[5] = 2;\\
&m [6] = 1; n [6] = 0; k[6] = 2; l[6] = 0; s[6] = 1; t[6] = 2; 
x[6] = 2; y[6] = 1;\\
&m [7] = 2; n [7] = 0; k[7] = 1; l[7] = 0; s[7] = 0; t[7] = 1; 
x[7] = 0; y[7] = 2;\\
&m [8] = 2; n [8] = 0; k[8] = 1; l[8] = 0; s[8] = 0; t[8] = 2; 
x[8] = 0; y[8] = 1;\\
&m [9] = 2; n [9] = 0; k[9] = 1; l[9] = 0; s[9] = 1; t[9] = 1; 
x[9] = 2; y[9] = 2;\\
&m [10] = 2; n [10] = 0; k[10] = 1; l[10] = 0; s[10] = 2; t[10] = 2; 
x[10] = 1; y[10] = 1;\\
&m [11] = 2; n [11] = 0; k[11] = 1; l[11] = 0; s[11] = 2; t[11] = 1; 
x[11] = 1; y[11] = 2;\\
&m [12] = 2; n [12] = 0; k[12] = 1; l[12] = 0; s[12] = 1; t[12] = 2; 
x[12] = 2; y[12] = 1;\\
&m [13] = 0; n [13] = 1; k[13] = 0; l[13] = 2; s[13] = 1; t[13] = 0; 
x[13] = 2; y[13] = 0;\\
&m [14] = 0; n [14] = 1; k[14] = 0; l[14] = 2; s[14] = 2; t[14] = 0; 
x[14] = 1; y[14] = 0;\\
&m [15] = 0; n [15] = 1; k[15] = 0; l[15] = 2; s[15] = 1; t[15] = 1; 
x[15] = 2; y[15] = 2;\\
&m [16] = 0; n [16] = 1; k[16] = 0; l[16] = 2; s[16] = 2; t[16] = 2; 
x[16] = 1; y[16] = 1;\\
&m [17] = 0; n [17] = 1; k[17] = 0; l[17] = 2; s[17] = 2; t[17] = 1; 
x[17] = 1; y[17] = 2;\\
&m [18] = 0; n [18] = 1; k[18] = 0; l[18] = 2; s[18] = 1; t[18] = 2; 
x[18] = 2; y[18] = 1;\\
&m [19] = 0; n [19] = 2; k[19] = 0; l[19] = 1; s[19] = 1; t[19] = 0; 
x[19] = 2; y[19] = 0;\\
&m [20] = 0; n [20] = 2; k[20] = 0; l[20] = 1; s[20] = 2; t[20] = 0; 
x[20] = 1; y[20] = 0;\\
&m [21] = 0; n [21] = 2; k[21] = 0; l[21] = 1; s[21] = 1; t[21] = 1; 
x[21] = 2; y[21] = 2;\\
&m [22] = 0; n [22] = 2; k[22] = 0; l[22] = 1; s[22] = 2; t[22] = 2; 
x[22] = 1; y[22] = 1;\\
&m [23] = 0; n [23] = 2; k[23] = 0; l[23] = 1; s[23] = 2; t[23] = 1; 
x[23] = 1; y[23] = 2;\\
&m [24] = 0; n [24] = 2; k[24] = 0; l[24] = 1; s[24] = 1; t[24] = 2; 
x[24] = 2; y[24] = 1;\\
&m [25] = 1; n [25] = 1; k[25] = 2; l[25] = 2; s[25] = 1; t[25] = 0; 
x[25] = 2; y[25] = 0;\\
&m [26] = 1; n [26] = 1; k[26] = 2; l[26] = 2; s[26] = 2; t[26] = 0; 
x[26] = 1; y[26] = 0;\\
&m [27] = 1; n [27] = 1; k[27] = 2; l[27] = 2; s[27] = 0; t[27] = 1; 
x[27] = 0; y[27] = 2;\\
&m [28] = 1; n [28] = 1; k[28] = 2; l[28] = 2; s[28] = 0; t[28] = 2; 
x[28] = 0; y[28] = 1;\\
&m [29] = 1; n [29] = 1; k[29] = 2; l[29] = 2; s[29] = 2; t[29] = 1; 
x[29] = 1; y[29] = 2;\\
&m [30] = 1; n [30] = 1; k[30] = 2; l[30] = 2; s[30] = 1; t[30] = 2; 
x[30] = 2; y[30] = 1;\\
&m [31] = 2; n [31] = 2; k[31] = 1; l[31] = 1; s[31] = 1; t[31] = 0; 
x[31] = 2; y[31] = 0;\\
&m [32] = 2; n [32] = 2; k[32] = 1; l[32] = 1; s[32] = 2; t[32] = 0; 
x[32] = 1; y[32] = 0;\\
&m [33] = 2; n [33] = 2; k[33] = 1; l[33] = 1; s[33] = 0; t[33] = 1; 
x[33] = 0; y[33] = 2;\\
&m [34] = 2; n [34] = 2; k[34] = 1; l[34] = 1; s[34] = 0; t[34] = 2; 
x[34] = 0; y[34] = 1;\\
&m [35] = 2; n [35] = 2; k[35] = 1; l[35] = 1; s[35] = 2; t[35] = 1; 
x[35] = 1; y[35] = 2;\\
&m [36] = 2; n [36] = 2; k[36] = 1; l[36] = 1; s[36] = 1; t[36] = 2; 
x[36] = 2; y[36] = 1;\\
&m [37] = 2; n [37] = 1; k[37] = 1; l[37] = 2; s[37] = 1; t[37] = 0; 
x[37] = 2; y[37] = 0;\\
&m [38] = 2; n [38] = 1; k[38] = 1; l[38] = 2; s[38] = 2; t[38] = 0; 
x[38] = 1; y[38] = 0;\\
&m [39] = 2; n [39] = 1; k[39] = 1; l[39] = 2; s[39] = 0; t[39] = 1; 
x[39] = 0; y[39] = 2;\\
&m [40] = 2; n [40] = 1; k[40] = 1; l[40] = 2; s[40] = 0; t[40] = 2; 
x[40] = 0; y[40] = 1;\\
&m [41] = 2; n [41] = 1; k[41] = 1; l[41] = 2; s[41] = 1; t[41] = 1; 
x[41] = 2; y[41] = 2;\\
&m [42] = 2; n [42] = 1; k[42] = 1; l[42] = 2; s[42] = 2; t[42] = 2; 
x[42] = 1; y[42] = 1;\\
&m [43] = 1; n [43] = 2; k[43] = 2; l[43] = 1; s[43] = 1; t[43] = 0; 
x[43] = 2; y[43] = 0;\\
&m [44] = 1; n [44] = 2; k[44] = 2; l[44] = 1; s[44] = 2; t[44] = 0; 
x[44] = 1; y[44] = 0;\\
&m [45] = 1; n [45] = 2; k[45] = 2; l[45] = 1; s[45] = 0; t[45] = 1; 
x[45] = 0; y[45] = 2;\\
&m [46] = 1; n [46] = 2; k[46] = 2; l[46] = 1; s[46] = 0; t[46] = 2; 
x[46] = 0; y[46] = 1;\\
&m [47] = 1; n [47] = 2; k[47] = 2; l[47] = 1; s[47] = 1; t[47] = 1; 
x[47] = 2; y[47] = 2;\\
&m [48] = 1; n [48] = 2; k[48] = 2; l[48] = 1; s[48] = 2; t[48] = 2; 
x[48] = 1; y[48] = 1;
\end{align*}

We use a ``for loop" to reduce $z$ symbolically:
\begin{align*}
&\text{For}[i=1,i<49,i\text{++}, \\ 
&\text{Print}[ i];\text{Print} [\text{Reduce}[\{z t(u^2-1) \varphi J'[1,2] ==f[\varphi J'[0,1],\varphi J'[1,1]], \\ 
&\text{  }z(u-1)\varphi J'[2,0]\text{==}f[\varphi J'[0,1],\varphi J'[2,2]],u\neq 0,\alpha \neq 0,\beta \neq 0,\gamma \neq 0, \\ 
&\delta \neq 0 \},z ] ] ].
\end{align*}

\subsection{Verifying the automorphism $\psi$ } \label{psi}
The map $\psi$ was defined on the basis of $J_3(1, z^2)$ as follows:
\begin{align*}
	&J''(1, 0) \mapsto - J'(1, 0),  && J''(2, 0) \mapsto - J'(2, 0), \\ & J''(0, 1) \mapsto -z J'(0, 2), && J''(0, 2) \mapsto - J'(0, 1), \\ & J''(1, 1) \mapsto \frac{z}{1 + u} J'(1, 2), && J''(2, 2) \mapsto \frac{1}{1 + u} J'(2, 1) \\ & J''(1, 2) \mapsto (1 + u) J'(1, 1) && J''(2, 1) \mapsto z(1 + u)J'(2, 2).
\end{align*}
Then 
{ \small
\begin{enumerate}[(1)]
	\item $\psi([J''(1, 0), J''(0, 1)]) = z(1 - u)J'(1, 2) = [\psi(J''(1, 0)), \psi(J''(0, 1))]$,
	\item $\psi([J''(1, 0), J''(0, 2)]) = (1 - u^2)J'(1, 1) = [\psi(J''(1, 0)), \psi(J''(0, 2))]$,
	\item  $
	\psi([J''(1, 0), J''(1, 1)]) = z(1 - u^2)(1 + u)J'(2, 2) = [\psi(J''(1, 0)), \psi(J''(1, 1))]
	$,
	\item  $
	\psi([J''(1, 0), J''(2, 2)]) = -(1 - u)J'(0, 2) = [\psi(J''(1, 0)), \psi(J''(2, 2))]
	$,
	\item  $
	\psi([J''(1, 0), J''(1, 2)]) = -(1 + u)(1 - u^2)J'(2, 1) = [\psi(J''(1, 0)), \psi(J''(1, 2))]
	$,
	\item  $
	\psi([J''(1, 0), J''(2, 1)]) = -z(1 - u^2)J'(0, 2) = [\psi(J''(1, 0)), \psi(J''(2, 1))]
	$,
	\item $\psi([J''(2, 0), J''(0, 1)]) = z(1 - u^2)J'(2, 2) = [\psi(J''(2, 0)), \psi(J''(0, 1))]$,
	\item  $\psi([J''(2, 0), J''(0, 2)]) = (1 - u)J'(2, 1) = [\psi(J''(2, 0)), \psi(J''(0, 2))]$,
	\item  $\psi([J''(2, 0), J''(1, 1)]) = -z(1 - u)J'(0, 2) = [\psi(J''(2, 0)), \psi(J''(1, 1))]$,
	\item  $\psi([J''(2, 0), J''(2, 2)]) = (1 - u^2)(1 + u)J'(1, 1) = [\psi(J''(2, 0)), \psi(J''(2, 2))]$,
	\item  $\psi([J''(2, 0), J''(1, 2)]) = -(1 - u^2)J'(0, 1) = [\psi(J''(2, 0)), \psi(J''(1, 2))]$,
	\item  $\psi([J''(2, 0), J''(2, 1)]) = -z(1 + u)(1 - u^2)J'(1, 2) = [\psi(J''(2, 0)), \psi(J''(2, 1))]$,
	\item  $\psi([J''(0, 1), J''(1, 1)]) = -z^2(u - 1)J'(2, 0) = [\psi(J''(0, 1)), \psi(J''(1, 1))]$,
	\item  $\psi([J''(0, 1), J''(1, 2)]) = -z^2(u^2 - 1)J'(2, 0) = [\psi(J''(0, 1)), \psi(J''(1, 2))]$,
	\item  $\psi([J''(0, 1), J''(2, 1)]) = -z^2(u^2 - 1)(1 + u)J'(2, 1) = [\psi(J''(0, 1)), \psi(J''(2, 1))]$,
	\item  $\psi([J''(0, 2), J''(1, 1)]) = -z^2(u - 1)J'(1, 0) = [\psi(J''(0, 2)), \psi(J''(1, 1))]$,
	\item  $\psi([J''(0, 2), J''(2, 2)]) = z(1 + u)(u^2 - 1)J'(2, 2) = [\psi(J''(0, 2)), \psi(J''(2, 2))]$,
	\item  $\psi([J''(0, 2), J''(1, 2)]) = -z(1 + u)(u^2 - 1)J'(1, 2) = [\psi(J''(0, 2)), \psi(J''(1, 2))]$,
	\item  $\psi([J''(0, 2), J''(2, 1)]) = -z^2(u^2 - 1)J'(2, 0) = [\psi(J''(0, 2)), \psi(J''(2, 1))]$,
	\item  $\psi([J''(1, 1), J''(1, 2)]) = z^2(u - u^2)J'(2, 0) = [\psi(J''(1, 1)), \psi(J''(1, 2))]$,
	\item  $\psi([J''(1, 1), J''(2, 1)]) = -z^2(u - u^2)J'(0, 1) = [\psi(J''(1, 1)), \psi(J'' (2, 1))]$,
	\item  $\psi([J''(2, 2), J''(1, 2)]) = z(u^2 - u)J' (0, 2) = [\psi(J''(2, 2)), \psi(J''(1, 2))]$,
	\item  $\psi([J''(2, 2), J''(2, 1)]) = -z^2(u^2 - u)J' (1, 0) = [\psi(J''(2, 2)), \psi(J''(2, 2))]$.
\end{enumerate}}
\end{document}